\documentclass[11pt]{article}

\usepackage{epsf,amsmath,amsfonts,amssymb,amsthm,graphicx,mathrsfs,latexsym,enumerate,leftidx,float}
\usepackage{epsfig} 
\usepackage[text={7.0in,9.5in},centering]{geometry} 
\usepackage{multirow}
\usepackage{color}

\newtheorem{thm}{Theorem}[section]
\newtheorem{lemma}[thm]{Lemma}

\newtheorem{prop}[thm]{Proposition}
\newtheorem{de}[thm]{Definition}
\newtheorem{cor}[thm]{Corollary}
 
\newtheorem{re}[thm]{Remark}

\newtheorem*{acknowledgement*}{Acknowledgement}

\makeatletter
\newcommand{\vast}{\bBigg@{4}}
\newcommand{\Vast}{\bBigg@{5}}
\makeatother

\begin{document}

\title{Markovian Switching of Mutation Rates in Evolutionary Network Dynamics}
\author{ Andrew Vlasic\\ Department of Mathematics and Statistics \\ Queen's University }
\date{}

\maketitle

\begin{abstract}
The replicator-mutator dynamic was originally derived to model the evolution of language, and since the model was derived in such a general manner, it has been applied to the dynamics of social behavior and decision making in multi-agent networks. For the two type population, a bifurcation point of the mutation rate is derived, displaying different long-run behaviors above and below this point. The long-run behavior would naturally be subjected to noise from the environment, however, to date there does not exist a model that dynamically accounts for the effects of the environment. To account for the environmental impacts on the evolution of the populace, mutation rates above and below this bifurcation point are switched according to a continuous-time Markov chain. The long-run behaviors of this model are derived, showing a counterintuitive result that the majority of initial conditions will favor the dominated type. 
\end{abstract}


\section{Introduction}
To describe the evolution of language in a population Nowak, Komarova, and Niyogi \cite{NKN01} applied an evolutionary game theory process, the replicator dynamic (see \cite{hofbauer1998evolutionary}), to display why languages spoken within a population eventually become extinct. As language is a construct and an actor in a subpopulation with a low frequency may perceive an advantage to speak another language, the authors restricted the payoffs to the unit interval and used the values to represent how well two agent may converse. To adjust for an actor immediately switching their language, a mutator parameter is introduce to dictate how often this switch occurs. 

Since the derivation of the mutator-replicator dynamic, there has been a natural application to the evolution of decision making and multi-agents networks. Applications includes the modeling of wireless multi-agents  \cite{TAEH10}, and the evolution of dominance in social networks \cite{OS07}.  The dynamic has proven to be quite complex, inhibiting analyzing and characterizing the model for a general number of nodes. To simplify the analysis, a fixed global mutation rate is assumed to dictate the evolution of the population, where the mutation rate for an agent to switch to a specific type is weighted using the associated payoffs. While the complexity of the model has made analyzing the long-run behavior quite difficult, a few authors have been able to determine fixed and stable points, and bifurcations about the mutation parameter when there are two or three network nodes \cite{KL10, PCL13}. 

The consideration of evolving populations in randomly varying environments has become quite profound in capturing natural events \cite{GGMP12,VL1}.  However, the replicator-mutator has yet to be considered in this environment. There are many examples displaying the importance of modeling such phenomena, for instance: social media contributed to individuals participating in the Arab Spring \cite{internet, opening} by manipulating the mutation rate between the two types of agents who are either participating or not-participating; the evolution of segregation within a society and the influence of social norms \cite{EF07, EHMM2016}; and the influence social interactions have with individuals updating or changing their opinions or beliefs \cite{opinion2013}. Taking the replicator-mutator under the influence of random environments would capture the evolution of networks with respect to social phenomena quite well.

To capture the influence the environment has on the global interaction of agents, a model is proposed that considers the replicator-mutator with a randomly changing value of the global mutation parameter, where the various environments dictate the values of the global mutation parameter and the switching of the mutation value is independent of agent interactions. Typical stochastic accounting incorporates a Gaussian forcing term which would override the internal fixed points of the deterministic process. As the stochastic forcing does not negate the influence of the the internal fixed points, the dynamic is further complicated, creating the potential for multiple long-run behaviors. To display the complexity that this minor change has on the long-run behavior of the dynamic, the case of a two subpopulation network and a random change of two environments is analyzed. The global long-run behavior is characterized by displaying local properties in natural subintervals and combined to show the dynamic is tight (defined in Section \ref{Long-Run Behavior} and see \cite{EK} for further reference), and the existence of subintervals where the invariant measure will have positive mass. Although the dynamic is tight, the existence of situations where the initial condition will dictate the long-run behavior will be shown. 

Unsurprisingly, given the initial condition where the dominant type in the deterministic setting has a high frequency, there is a potential for the stochastic perturbation to favor the dominated type. However, if the dominant type has a relatively low frequency, which still may be larger than the dominated type, the evolution of this population will favor the dominated type. 

Lastly, the paper ``On Equilibrium Properties of the ReplicatorÐMutator Equation in Deterministic and Random Games" \cite{mr_stoch} was the first to consider noise in the dynamic by assuming the game is random. The authors utilize techniques from classical and random polynomial theory to assist in deriving the expected number of equilibria. This differs significantly from this paper since the game is assumed to be fixed and the evolution of this dynamic is analyzed to established the expected long-run behavior given any initial conidition. 

\subsection{Preliminary on Markov Chains}
Before the model is introduced we shall briefly go over necessary definitions of Markov chains, the Markov property, and martingales. When natural, a definition will be tailored to the model, introduced in the following section, which is an ordinary differential equation composed with a continuous time Markov chain. 

Let $r = \big( r(t) \big)_{t \geq 0}$ be a continuous-time Markov chain with the state space $\{1,2\}$ and generator
$
Q=
\left(
\begin{array}{cc}
 -q_{12} & q_{12}     \\
 q_{21} & - q_{21} 
\end{array}
\right),
$
where $q_{12}>0$ is a transition rate from state 1 to state 2, and $q_{21}>0$ is a transition rate from state 2 to state 1. For instance, taking $\delta>0$ small enough, $P\Big( r(t+\delta)=2 \Big| r(t)=1 \Big)= q_{12} \delta + o(\delta)$ and $P\Big( r(t+\delta)=1 \Big| r(t)=2 \Big)= q_{21} \delta + o(\delta)$.  

To account for the random times when events in the Markov chain process occur, the definition below introduces when such times are well-defined.

\begin{de}
A random time $\tau$ is called a stopping time with respect to a filtration $\big\{ \mathcal{F}_t\big\}_{t\in\mathbb{R}^+}$ if for all $t$, $\{\tau \leq t\} \in \mathcal{F}_t$. Notice that if $\tau$ is a stopping time then $\tau \wedge t:=\min\{\tau,t\}$ is a finite stopping time for all $t$.
\end{de}

Given an increasing sequence of times that the process jump to a different state, which are also stopping times, say $0=\tau_0<\tau_1< \ldots <\tau_k \to \infty$, $r(t)$ may be written as
$$
r(t) =\sum_{ k=0}^{\infty} r( \tau_k) I_{ [\tau_k, \tau_{k+1}) } (t),
$$
where $I_{[s_1,s_2)}(\cdot)$ is the indicator function on the time interval $[s_1,s_2)$. Given that $ r( \tau_k)=i$, the times between jumps are assumed to be exponentially distributed with parameter $q_{ij}$, with $j:=3- i$. Throughout this paper, $j=3- i$.  Hence, for any $T \geq 0$, $P\Big( \tau_{k+1} - \tau_{k} \geq T \Big| r(\tau_{k})=1 \Big) = e^{ -q_{12} T}$, and similarly for the initial condition $r(\tau_{k})=2$. Finally, this class of Markov chain has a unique stationary distribution $\Pi = \big( \pi_1, \pi_2\big)$, where $\displaystyle \pi_1 = \frac{q_{21} }{ q_{12} +q_{21} }$ and $\displaystyle \pi_2 = \frac{q_{12} }{ q_{12} +q_{21} }$.

The model is an ordinary differential equation with randomly changing parameters subject to a continuous time Markov chain, say $X(t)$. The evolution of such a process has the potential to be conditioned off of the current state of the process,  $X(0)=x_0$, and the state of the Markov chain, $r(t)=i$. The initial conditions will be written as the tuple, $(x_0,i)$. Please see \cite{YZ10} for further information pertaining to the switching of deterministic systems. To keep generality, the following definitions only assume a composition of process.

Take $X(t)$ as a continuous-time Markov process  dependent on time and the state of the Markov process $r(t)$, and define $P_{x_0,i}$ as the probability measure corresponding to $X(t)$ when $X(0)=x_0$ and $r(0)=i$. Throughout this paper, the initial conditions are assumed to almost surely hold (with probability one). Furthermore, define $E_{x_0,i}$ as the expectation taken with the initial condition $(x_0,i)$, hence, with respect to the measure $P_{x_0,i}$.

\begin{de}  
For the following definitions, take the pair $(x_0,i)$ as the initial condition, and $I$ as the state space of the process $r(t)$. Furthermore, take the process to evolve strictly in the unit interval.
\begin{itemize}
	\item A Markov process $X(t)$ is said to be regular if for any finite time $T>0$,
$$
P_{x_0,i} \bigg( \inf_{0 \leq t \leq T} X(t) = 0 \ \textnormal{or} \ \sup_{0 \leq t \leq T} X(t) =1 \bigg)=0.
$$
	\item Define $\displaystyle U:= D \times J$, where $D \subsetneq [0,1]$ nonempty and $J \subseteq I$, and $\displaystyle \tau_U : = \inf\Big\{ t \geq 0 : X(t) \in U \Big\}$. A Markov process $X(t)$ is called recurrent with respect to $U$ if it is regular and for any finite time $T>0$,
$$
P_{x_0,i} \Big( \tau_{U} < \infty \Big)=1,
$$
for an arbitrary initial condition $ (x_0,i) \in D^c \times I$. (The notation $D^c$ denotes the compliment of the set $D$.)
	\item If $X(t)$ is not recurrent, it is called transient. 
	\item The process is called positive recurrent with respect to the set $U$ if it is recurrent and $D$ is a strict nonempty subset of the unit interval.
\end{itemize}
\end{de}

This paper utilizes the well applied strong Markov property and Dynkin's formula of probability theory. The definition and formula are established below.

The \textbf{strong Markov property} states that, for a stopping time $\tau$ where $P_{x_0,i}\big( \tau<\infty \big)=1$, the equality $X(t+\tau)-X(\tau)=X(t)$ holds in distribution. Using this definition, \textbf{Dynkin's formula} states that for a function $f(\cdot,i)\in C^2$ for all $i$ and a stopping time $\tau$ where $E_{x_0,i}\big[\tau\big]<\infty$, we have that
$$
E_{x_0,i}\big[f\big(X(\tau)\big)\big]= f\big(x_0,i\big) + E_{x_0,i}\bigg[ \int_{0}^{\tau} Lf\big( X(s) \big)ds\bigg].
$$

\begin{de}
For $0 \leq s \leq t$ and a Markov process $X$, define $\displaystyle E\big[ X(t) \big| \sigma\{ X(s_0) : s_0 \leq s\}  \big]$ as the conditional expectation of the process at time $t$ dependent on the history of the process up to time $s$, (where $\sigma\{ X(s_0) : s_0 \leq s\}$ is the $\sigma$-algebra of the process up to time $s$). A stochastic process $X$ is called a \textbf{martingale} if $\displaystyle E\big[ X(t) \big| \sigma\{ X(s_0) : s_0 \leq s\}  \big]= X(s)$, a \textbf{submartingale} if $\displaystyle E\big[ X(t) \big| \sigma\{ X(s_0) : s_0 \leq s\} \big] \geq X(s)$, and a \textbf{supermartingale} if $\displaystyle E\big[ X(t) \big| \sigma\{ X(s_0) : s_0 \leq s\}  \big] \leq X(s)$.
\end{de}

Lastly, the definitions below give the properties to describe the evolution of the dynamic and understand the long-run behavior of the dynamic. 

\begin{de}
The stationary point $x=0$ is said to be:
\begin{enumerate}
\item stable in probability if for any $\epsilon>0$ and any $i$,
$$
\lim_{x_0 \to 0} P_{x_0,i}\Big( \sup_{t \geq 0} X(t) > \epsilon \Big) =0;
$$
\item stochastically stable if it is stable in probability and
$$
\lim_{x_0 \to 0} P_{x_0,i}\Big( \lim_{t \to \infty} X(t)  = 0 \Big) =1;
$$
\item stochastically unstable if it is not stochastically stable. 
\end{enumerate}
\end{de}

\section{Results}

\subsection{Replicator-Mutator Dynamic}\label{MR}
The evolution of behavior adaptation in social networks with the influence of a changing environment is considered. For the general setting, there are $n$ types (or nodes), where agents of type $i$ has behavior (or strategy) $S_i$. Define $b_{ij}$ as the payoff for an agent with behavior $S_i$ interacting with $S_j$. The model assumes $b_{ii}=1$, and for $j\neq i$, $0\leq b_{ij}<1$. With these payoffs, an agent would prefer to interact with another agent with the same behavior, however, there are other behaviors that the agent has a positive interaction. This general game is known as a coordination game, and the frequency of a type tends to dictate the evolution of the population. If $b_{ij}>0$, we say that type $i$ is \textbf{attracted} to behavior $S_j$.  

Define $\displaystyle B=(b_{ij})_{1\leq i,j \leq n}$ as the payoff matrix, take $x_i\in[0,1]$ to be the frequency of type $i$, and $\mathbf{x}:=(x_1, x_2, \ldots, x_n)$. Note that the payoff matrix $B$ has a graph theoretic interpretation of the adjacency matrix of a directed graph. The natural assumption that interactions of agents are random, well-mixed, independent, and identically distributed is taken for the model. Under this assumption, the fitness of an agent in type $i$ is $(B\mathbf{x})_i$. Classic replicator dynamics states the the evolution of type $i$ is given by the dynamic $\displaystyle \dot{x}_i =  x_i \bigg( \big(B \mathbf{x}\big)_i - \mathbf{x}\cdot B \mathbf{x} \bigg)$, which assumes growth is proportional to difference of the fitness of type $i$ and the average fitness of the population. 
 
Since social networks are considered, there exists the distinct potential of an agent switching their behavior. For instance, if an agent of type $i$ has a smaller fitness than that of an agent of type $j$, and $S_j$ is a behavior that an agent in type $i$ is attracted to, it would be advantageous for this agent to switch behaviors. For a constant $0\leq \mu \leq 1$, the rate of type $i$ adopting the behavior of type $j$ is given by 
$$
\displaystyle p_{ij}:=
\left\{
\begin{array}{cc}
\frac{\mu b_{ij} }{ \sum_{k\neq i } b_{ik}  }  &   i\neq j   \\
1-\mu  & \mbox{otherwise}      \\
   
\end{array}
\right. .
$$ 
From these mutation rates, define the matrix of mutations as $\displaystyle P=(p_{ij} )_{1\leq i,j \leq n}$, which is a stochastic matrix.  From this structure, the replicator-mutator dynamic for type $i$ is defined as
$$
\dot{x}_i = \sum_{j=1}^n p_{ji} x_j \big(B \mathbf{x}\big)_j - x_i  \sum_{j=1}^n  x_j \big(B \mathbf{x}\big)_j.
$$
Notice, when $\mu=0$ the replicator-mutator is the classic replicator dynamic, and when $\mu=1$ agents will continuously adopt another behavior. Because each entry in $B$ is positive, the dynamic is invariant to the $(n-1)$-dimensional simplex. 

When there are two strategies, with $b_1:=b_{12}$ and $b_2:=b_{21}$, the dynamic simplifies to 
$$
\dot{x}_1= x_1\Big[ b_1 + x_1(1-b_1)\Big] \big( 1-\mu -x_1\big)   +   \big( 1-x_1 \big)\Big[ 1 + x_1(b_2 -1)\Big] \big( \mu - x_1 \big)
$$
(see \cite{PL11}). This particular dynamic has a bifurcation about the parameter $\mu$ (see \cite{KL10, PCL13}), and to account for the bifurcation denote $\mu_c$ as the critical point.  

Consider the case when $b_1 \neq b_2$. When $\mu \leq \mu_c$, there are three fixed point, say $a_1 < a_2< a_3$, where $a_1$ and $a_3$ are stable, and $a_2$ is unstable. When $\mu>\mu_c$, there is one real fixed point $\hat{a}$ that is globally stable. For reference of the size of the globally stable point $\hat{a}$, when $b_2 > b_1$ we see  $a_1<\hat{a} < a_2$, and when $b_2 < b_1$ we see $a_2<\hat{a} < a_3$. While this difference in the size of $\hat{a}$ will have an effect the evolution of the dynamic, the proofs of the properties will apply the same techniques and logic, and hence will be stated as corollaries.

When $b_1=b_2$ the symmetry creates a pitchfork bifurcation. This pitchfork bifurcation implies $\hat{a} = a_2$, and the bifurcation point has the simple form $\displaystyle \mu_c= \frac{1-b_1}{4}$.

\subsection{Markovian Switching of the Mutation Parameter}\label{MSmodel}
To account for stochastic forces affecting the population, fixed values of the mutation parameter will be randomly switched during the evolution of the dynamic. The stochastic force is assumed to be generated by a continuous-time Markov chain that is independent of the population dynamic. To develop intuition only two values of $\mu$ will be considered. We call these two values $\mu_1$ and $\mu_2$ where $\mu_1<\mu_c<\mu_2$. 

Define $x(t)$ as the replicator-mutator dynamic with with the mutator randomly switched, and take $r(t)$ as the continuous-time Markov chain that dictates the switching of the parameter. From the description of the dynamic, $r(t)$ stays in state space $\{1,2\}$ where the generator of $r(t)$, denoted as $Q$, has the form
$
Q=
\left(
\begin{array}{cc}
 -q_{12} & q_{12}     \\
 q_{21} & - q_{21} 
\end{array}
\right)
$. 

The dynamic $x(t)$ evolves according to
\begin{equation}\label{smr_first}
dx\big( t \big) = x{(t)} \Big[ b_1 + x{(t)} (1-b_1)\Big] \big( 1-\mu_{r(t)} -x{(t)} \big)   +   \big( 1-x{(t)}  \big)\Big[ 1 + x{(t)}  (b_2 -1)\Big] \big( \mu_{r(t)} - x{(t)}  \big)dt,
\end{equation}
where $r(t)$ evolves according to
\begin{equation}\label{mcr}
P\Big(  r\big( t + \Delta t \big) = j \big| r\big( t \big) =i  \Big) = q_{ij}(\Delta t) + o(\Delta t), \ \ i \neq j,
\end{equation}
and we collect both processes as
\begin{equation}\label{smr}
 Y(t) := \big( x(t), r(t) \big).
\end{equation}

One may see that $Y(t)$ is a two-component Markov process, and Section \ref{General Behavior} establishes the process to be well-defined. The generator for this random process, denoted as $L$, has the form
\begin{equation}\begin{split}
LV(x,i) & = V'(x,i) \cdot \Big\{ x \Big[ b_1 + x (1-b_1)\Big] \big( 1-\mu_i -x \big)   +   \big( 1- x \big)\Big[ 1 + x (b_2 -1)\Big] \big( \mu_i- x \big) \Big\} \\
&+ \Big[ q_{ij}V(x,j) - q_{ij}V(x,i) \Big] \\
& = V'(x,i)\cdot \dot{Y}(x,i) + \Big[ q_{ij}V(x,j) - q_{ij}V(x,i) \Big],
\end{split}\end{equation}
where the function $V(x,i)$ is twice continuously differentiable for $x$ and $i=1,2$, (see \cite{YZ10} for further information). 

The first term has the classical Lyapunov form, while the second part represents the contribution of the perturbations from the switching of the mutator parameter. Adapting to the notation, define the dynamic $\displaystyle \dot{Y}(x,i)= x \Big[ b_1 + x (1-b_1)\Big] \big( 1-\mu_{i} -x \big)   +   \big( 1-x \big)\Big[ 1 + x (b_2 -1)\Big] \big( \mu_{ i} - x \big)$.

\section{Long-Run Behavior}\label{Long-Run Behavior}

To establish the results, the stochastic Lyapunov method will be utilized to show the expected hitting time for the process to hit particular closed subintervals within the unit interval are finite. The direction for which the process hits the subinterval will play a role in the evolution of the dynamic, and to adjust for this subtly, the hitting times will be denoted for coming from the left or from the right. Define the hitting time $\tau^-_{x_0}$ for the process coming from the left and the left endpoint of the subinterval is $x_0$, and $\tau^+_{x_0}$ as the hitting time for the process coming from the right and the right endpoint of the subinterval is $x_0$. 

After Lemma \ref{0trans} establishes the proof technique, for brevity any sequential results, after an inequality has been derived, it will be stated that the hitting time is finite. In particular, Lemma \ref{0trans} will display the inequality of  $LV(x,i) < -\kappa$ for a positive constant $\kappa$, $i=1,2$, and any $x$ in the subinterval of interest is established, where it will naturally follow that the expectation of the hitting time is finite; (see \cite{YZ10, MY06, MT93} for further information).  

Recall in Section \ref{MR}, when $\mu<\mu_c$ the values $a_1,a_2$, and $a_3$ are the only fixed points, and when $\mu>\mu_c$ the value $\hat{a}$ is the only real fixed point. Furthermore, for the case when $\mu>\mu_c$, if $b_2 > b_1$ then $\hat{a} < a_2$, and if $b_2 < b_1$ then $\hat{a} > a_2$.  This observation may be readily seen from the equality $Y\big( x_2(t), r(t) \big) = 1 - Y\big( x_1(t), r(t) \big)$. Finally, when $b_1=b_2$ there is pitchfork bifurcation which yields $\hat{a} = a_2$.


Throughout this section, without loss of generality, the state $i=1$ corresponds to $\mu<\mu_c$, and the state $i=2$ corresponds to $\mu>\mu_c$. 

\subsection{General Behavior}\label{General Behavior}
While one may observe that the evolution of the dynamic differs completely between the two cases of $b_1\neq b_2$ and $b_1 = b_2$, as will be seen in this subsection, this distinction dictates the global evolution of the dynamic and not the local evolution near the points $a_1$ and $a_3$. Hence, the results pertain to the general dynamic defined in Section \ref{MSmodel}.

The initial conditions of $i=1,2$  and $x_0\in(0,a_1)$ then $x_0\in(a_3,1)$ are first considered to display the dynamic to evolve away from the end points of the unit interval. Whether the deterministic dynamic is in state $i=1$ or $1=2$, the process will flow away from each end point thus showing that lemmas \ref{0trans} and \ref{1trans} are intuitive.

\begin{lemma}\label{0trans}
For the dynamic defined by Equation \eqref{smr} with the initial conditions $x_0\in(0,a_1)$ and $i \in \{1,2\}$, $E_{x_0,i}\big[ \tau^-_{a_1} \big]<\infty$.
\end{lemma}
\begin{proof}
Define $V(x,i)=-c_i \log(x)$, and note that $\dot{Y}(x,1) > 0$ for $x\in(0,a_1)$, $\dot{Y}(a_1,1) =0$, and $\dot{Y}(x,2) > 0$ for $x\in(0,a_1]$. Then 
\begin{equation*}\begin{split}
LV(x,i) = \frac{ -c_i \dot{Y}(x,i)}{x} +q_{ij} \big( c_i -c_j \big) \log(x).
\end{split}\end{equation*}
Taking $c_1>c_2>0$, and $c_1 - c_2 <<1$, there exists a $\kappa>0$ such that $\displaystyle LV(x,i)<-\kappa$, for all $x\in(0,a_1)$. Thus, by Dynkin's formula, we have that
\begin{equation*}\begin{split}
0 \leq E_{x_0,i}\Big[ V\big(X( \tau^-_{a_1} \wedge t)  \big)\Big] & = V(x_0,i) + E_{x_0,i}\bigg[  \int_{0}^{ \tau^-_{a_1} \wedge t } LV\big(X(s)\big)ds\bigg] \\
& \leq V(x_0,i) -\kappa E_{x_0,i}\big[ \tau^-_{a_1} \wedge t\big]. 
\end{split}\end{equation*}
Hence $\displaystyle E_{x_0,i}\big[ \tau^-_{a_1} \wedge t\big] \leq \frac{V(x_0,i)}{\kappa }$, and the monotone convergence yields our result.
\end{proof}

\begin{lemma}\label{1trans}
For the initial conditions $x_0\in(a_3,1)$ and $i \in \{1,2\}$, $E_{x_0,i}\big[ \tau^+_{a_3} \big]<\infty$.
\end{lemma}
\begin{proof}
Define $V(x,i)=-c_i \log(1-x)$, and note that $\dot{Y}(x,1) < 0$ for $x\in(a_3,1)$, $\dot{Y}(a_3,1) =0$, and $\dot{Y}(x,2) < 0$ for $x\in[a_3,1)$. Then 
\begin{equation*}\begin{split}
LV(x,i) = \frac{ c_i \dot{Y}(x,i)}{1-x} +q_{ij} \big( c_i -c_j \big) \log(1-x).
\end{split}\end{equation*}
Taking $c_1>c_2>0$, and $c_1 - c_2 <<1$, there exists a $\kappa>0$ such that $\displaystyle LV(x,i)<-\kappa$, for all $x\in(a_3,1)$. Therefore, following the logic of Lemma \ref{0trans}, our result follows.
\end{proof}

The previous lemmas established that the evolution of the dynamic will tend to the interval $(a_1,a_3)$, which yields the question: given an initial condition $x_0\in(a_1,a_3)$, will the dynamic stay in this interval? For simplicity, consider the situation where $a_1< x_0 <\hat{a} < a_2$, or the case when $b_2 < b_1$. In State 1, the dynamic will evolve to $a_1$, which has a slow convergence time to this fixed point. This slowly converging evolution will give time to randomly switch to State 2, changing the evolution of the dynamic to the point $\hat{a}$. A similar situation occurs for the dynamic close to $\hat{a}$. These heuristics display a natural characteristic of the dynamic to be invariant to the interval $(a_1,a_3)$. To display this property, the following two lemmas show for an initial condition in the interval, the points $a_3$ and $a_1$ are unstable. 


\begin{lemma}\label{a3trans}
For the dynamic defined by Equation \eqref{smr} with the initial conditions $x_0<a_3$ and $i\in\{1,2\}$, the point $x= a_3$ is stochastically unstable.
\end{lemma}
\begin{proof}
The argument below will utilize the supermartingale property coupled with the soon to be defined stochastic Lyapunov function exploding at $x=a_3$. A further exposition of this technique is found in Proposition 8.7 on page 226 in Yin and Zhu \cite{YZ10}. 

Define the Lyapunov function $V(x,i)=(1-\alpha c_i)\big(a_3-x\big)^{-\alpha}$. Then
$$
LV(x,i)=\bigg\{  \frac{\dot{Y}(x,i) }{ a_3-x } + q_{ij} \frac{ c_i -c_j }{ 1 -\alpha c_i}\bigg\}\Big(\alpha V(x,i) \Big).
$$
One may see that $\displaystyle \frac{\dot{Y}(x,1) }{  a_3-x }$ is bounded and positive, and $\displaystyle \frac{\dot{Y}(x,2) }{  a_3-x }$ is unbounded and negative. Thus, there exists a neighborhood of $a_3$ (contingent on the values of $c_1$, $c_2$, and $\alpha$), which we call $(a_3-\epsilon,a_3)$, such that $LV(x,i)<0$. Therefore, following the proof in Proposition \ref{prop-pos-rec}, $x=a_3$ is unstable. 
\end{proof}

\begin{lemma}\label{a1trans}
For the dynamic defined by Equation \eqref{smr} with the initial condition $x_0>a_1$ and $i\in\{1,2\}$, the point $x= a_1$ is stochastically unstable.
\end{lemma}
\begin{proof}
Define $V(x,i)=(1 + \alpha c_i ) (x-a_1)^{-\alpha}$, where $\alpha>0$, and $1-\alpha c_i>0$ for $i=1,2$. Then 
$$
LV(x,i) = \bigg\{  \frac{- \dot{Y}(x,i) }{ x-a_1 } + q_{ij} \frac{ c_i -c_j }{ 1 + \alpha c_i}\bigg\}\Big(\alpha V(x,i) \Big).
$$
Note that $\displaystyle \frac{- \dot{Y}(x,1) }{  x -a_1 }$ is bounded and positive, and $\displaystyle \frac{- \dot{Y}(x,2) }{  x -a_1 }$ is unbounded and negative. By the properties of $\displaystyle \frac{- \dot{Y}(x,2) }{  x -a_1 }$, one may find a $c_1$, $c_2$, and $\alpha$ so that $LV(x,1)<0$ in a neighborhood of $a_1$ \Big(intersected with $(a_1, \hat{a})$\Big). For completeness, call this neighborhood $(a_1,a_1+\epsilon_1)$ where $\epsilon_1>0$ is of appropriate size. Therefore, the point $x=a_1$ is unstable.
\end{proof}

\begin{re}
Lemmas \ref{0trans}, \ref{1trans}, \ref{a3trans}, and \ref{a1trans} display that the process will almost surely stay in the unit interval for all time. Therefore, the dynamic in Equation \eqref{smr} is well-defined.  
\end{re}

The previous lemmas display the evolution of the dynamic will naturally be restricted to some subinterval of $(a_1,a_3)$. Since the dynamic is invariant to a compact space, (mapping the state space to $\mathbb{R}$ will preserve the compact property), the dynamic should be \textit{\textbf{tight}}; recall that for a dynamic to be \textit{tight}, for any $\epsilon>0$ and initial condition $(x_0, i)$, there exists a compact set $K$ such that the transition probability $p_{x_0,i} \big( t, K \times \{1,2\} \big) \geq 1 - \epsilon $, for all time $t$. 

To display that the dynamic is tight, the existence of every moment will be shown. Utilizing the evolution displayed by the previous lemmas will yield this property. After the existence of every moment has been established, the tight property comes as a natural consequence.

\begin{prop}\label{moments}
For the dynamic defined by Equation \eqref{smr}, $\displaystyle \sup_{0 \leq t < \infty}E_{x_0,i} \big[  Y^p(t) \big] < \infty$. 
\end{prop}
\begin{proof}
For an arbitrary $p > 0$, define $g(x,i)=x^p$ 
\begin{equation*}\begin{split}
Lg(x,i) & =px^{p-1} \Big\{ x \big[  b_1 + x (1-b_1) \big]  \big( 1 - \mu_i  - x  \big)   +   \big( 1- x \big) \big[ 1 + x (b_2 -1) \big] \big(  \mu_i -  x  \big) \Big\} +\Big[ q_{ij}g(x,j) - q_{ij}g(x,i) \Big] \\
& = - \Big( p x\big[ b_1 + x (1-b_1)\big] + p(1-x) \big[ 1 + x (b_2 -1)\big]  \Big) x^p \\
& +  \Big( px \big( 1- \mu_i \big)  \big[  b_1 + x (1-b_1) \big]    +  p\mu_i\big(  1 -  x \big) \big[ 1 + x (b_2 -1) \big]  \Big) x^{p-1}.
\end{split}\end{equation*}

By Lemmas \ref{0trans}, \ref{1trans}, \ref{a3trans}, and \ref{a1trans}, the dynamic is almost surely transient to and invariant to the interval $(a_1,a_3)$. By the strong Markov property, we may take the initial condition $x_0 \in (a_1,a_3)$. Define $\displaystyle C_p = \min_{a_1< x < a_3} p x\big[ b_1 + x (1-b_1)\big] + p(1-x) \big[ 1 + x (b_2 -1)\big] $ , $\displaystyle B_{p,i} = \max_{a_1<x<a_3}  \Big( px \big[  b_1 + x (1-b_1) \big] \big( 1- \mu_i \big)    +  p\mu_i\big(  1 -  x \big) \big[ 1 + x (b_2 -1) \big]  \Big) x^{p-1} $, and $B_p=\max\{ B_{p,1}, B_{p,2} \}$. One may see that $C_p$ and $B_p$ are strictly positive and bounded, and $Lg(x,i) \leq - C_p g(x,i) + B_p$. 

For $0<\lambda<C_p$, $L \big( e^{\lambda t} g(x,i) \big) = e^{\lambda t} L g(x,i) + \lambda e^{\lambda t} g(x,i) \leq  e^{\lambda t} B_p$, which implies $E_{x_0,i} \big[  x^p(t) \big] \leq e^{-\lambda t}x_0 + \frac{B_p}{\lambda} \left( 1 -  e^{-\lambda t} \right)$, and therefore $\displaystyle \sup_{0 \leq t < \infty}E_{x_0,i} \big[  x^p(t) \big] < \infty$.  
\end{proof}

\begin{cor}
The dynamic defined by Equation \eqref{smr} is tight.
\end{cor}
\begin{proof}
An application of the well-known Chebyshev inequality will display this property. As this is fairly straight forward to observe this characteristic, the proof is omitted.
\end{proof}

Now that it has been established that the dynamic is invariant to and recurrent in the subinterval $(a_1,a_3)$,  one must consider the fixed points inside this state space, where the pull from the fixed points and the time to switch from each state will drive the long-run behavior. There are three cases to be explored: $b_2 > b_1$, $b_2 < b_1$, and $b_2 = b_1$. 

\subsection{The Inequality Case}\label{The Inequality Case}
In this section the long-run behavior of the case when $b_2 \neq b_1$ is established. Since the cases of $b_2 > b_1$ and $b_2 < b_1$ differ only by the size of the globally stable point, applying the logic created from one of the cases will be enough to establish results for both. The propositions will be established with the assumption $b_2 > b_1$ then corollaries are extended to $b_2 < b_1$. 

\begin{prop}\label{b2<b1trans}
For the dynamic defined by Equation \eqref{smr}, if $b_2 > b_1$ with the initial condition $x_0\in(\hat{a},a_2)$ and $i \in \{1,2\}$ then $E_{x_0,i}\big[ \tau^+_{\hat{a} }\big]<\infty$.
\end{prop}
\begin{proof}
The proof is identical to the one given in Lemma \ref{1trans}, with the Lyaponuv function $V(x,i)=-c_i \log(a_2-x)$.
\end{proof}

For the case $b_2 < b_1$, Proposition \ref{b2>b1trans} gives a similar scenario of the evolution of the dynamic except the initial condition would be in interval $(a_2, \hat{a})$, which would reverse the flow of the dynamic. Utilizing the logic in the proposition above, this ``reverse flow'' may be be shown by the Lyapunov function $V(x,i)=-c_i \log(x - a_2)$.  This intuition gives the following corollary.

\begin{cor}\label{b2>b1trans}
For the dynamic defined by Equation \eqref{smr}, if $b_2 < b_1$ with the initial condition $x_0\in(a_2, \hat{a})$ and $i \in \{1,2\}$ then $E_{x_0,i}\big[ \tau^-_{\hat{a} }\big]<\infty$.
\end{cor}

For $b_2 > b_1$ and initial condition $x\in (a_1, \hat{a})$ and $i \in \{1,2\}$, it will be shown that $\displaystyle P_{x,i} \Big( Y(t) \in (a_1, \hat{a}) \Big)=1$, and as a consequence the process is recurrent in a  subinterval $(a_1, \hat{a})$. To display this property, the endpoints $a_1$ and $\hat{a}$ will be shown to be unstable by applying the stochastic Lyapunov method. For a heuristic proof, consider the dynamic with the initial condition in the interval $(a_1, \hat{a})$ and $i=1$. Under these conditions, the process will evolve to the point $a_1$, but as the process gets closer to this point it slows down, giving the Markov chain enough time to jump to state 2 and flow away from $a_1$. In turn, this will show that there is a strict subinterval of $(a_1, \hat{a} )$ that the dynamic will hit in finite time. The strong Markov property establishes the recurrent property. For $b_2 < b_1$, applying the same heuristics one may see a similar property holds within the interval $(\hat{a}, a_3)$.

\begin{thm}\label{prop-pos-rec}
For the dynamic defined by Equation \eqref{smr} where $b_2 > b_1$ with initial conditions $x_0\in (a_1, \hat{a})$ and $i \in \{1,2\}$, the dynamic is invariant to the interval $(a_1, \hat{a})$, i.e., $\displaystyle P_{x_0,i} \Big( Y(t) \in (a_1, \hat{a}) \Big)=1$. Furthermore, the process is recurrent in $(a_1, \hat{a})$.
\end{thm}
\begin{proof}
To prove this proposition, the points $a_1$ and $\hat{a}$ will be shown to be unstable. To begin the proof, note that under the assumptions, the inequalities $\dot{Y}(x,1) \leq 0$ and $\dot{Y}(x,2) \geq 0$ hold for $x\in (a_1, \hat{a})$.  Lemma \ref{a1trans} gives the property that $a_1$ is unstable.

For the point $x=\hat{a}$, define Define $V(x,i) = (1-\alpha c_i ) (\hat{a} - x)^{-\alpha}$, where $\alpha>0$, and $1-\alpha c_i>0$ for $i=1,2$. Then
$$
LV(x,i) = \bigg\{  \frac{ \dot{Y}(x,i) }{ \hat{a} - x } + q_{ij} \frac{ c_i -c_j }{ 1 -\alpha c_i}\bigg\}\alpha V(x,i).
$$
Since $\displaystyle \frac{ \dot{Y}(x,2) }{ \hat{a} - x }$ is bounded, one may find a $c_1$, $c_2$, and $\alpha$ so that $LV(x,2)<0$ in a neighborhood of $a_1$ \Big(intersected with $(a_1, \hat{a})$\Big). Furthermore, since $\displaystyle \frac{ \dot{Y}(x,1) }{ \hat{a} - x }$ is unbounded, whatever the values of $c_1$ and $c_2$, there exist a neighborhood of $a_1$ such that $LV(x,1)<0$, say $(\hat{a}-\epsilon_2,\hat{a})$, where $LV(x,i)<0$ for both values of $i$. Thus, the point $x=\hat{a}$ is unstable. 

To show the recurrent property, taking an initial condition $x\in(a_1,a_1+\epsilon_1)$ and $i=1,2$, or  $x\in(\hat{a}-\epsilon_2,\hat{a})$  and $i=1,2$, for the Lyapunov function defined above and following the logic in Lemma \ref{0trans}, we may conclude that $\displaystyle E_{x,i}\big[\tau_{a_1+\epsilon_1}^-]<\infty$ and $\displaystyle E_{x,i}\big[\tau_{\hat{a}-\epsilon_2}^+]<\infty$. Thus, the process hits the interval $(a_1+\epsilon_1 ,  \hat{a}-\epsilon_2)$ in finite time, and therefore by the strong Markov property, the process is recurrent.  
\end{proof}

\begin{cor}\label{cor-pos-rec}
For the dynamic defined by Equation \eqref{smr} where $b_2 < b_1$ with initial conditions $x_0\in (\hat{a}, a_3)$ and $i \in \{1,2\}$,  $\displaystyle P_{x_0,i} \Big( Y(t) \in (\hat{a}, a_3 ) \Big)=1$ and the dynamic is recurrent in $(\hat{a}, a_3 )$.
\end{cor}

With the initial condition $x_0\in(a_2,a_3)$, recall that State 1 corresponds to the dynamic with the mutator term with a smaller value than the critical point, and State 2 is the dynamic with the mutator term with a larger value than the critical point. If the transition probability towards the second state is large enough then the expected time of hitting the interval discussed in Proposition \ref{b2<b1trans} is finite. Ergo, the evolution of the dynamic is quite sensitive to the values of the Markov generator. This sensitivity dictates the long-run characteristic of the process, hence, where the dynamic will evolve in the subinterval of $(a_2,a_3)$ or $(a_1,\hat{a})$, accordingly.

\begin{thm}\label{trans-prop}
For the dynamic defined by Equation \eqref{smr} where $b_2 > b_1$, then there exists an $\epsilon>0$ such that for an initial condition $x_0\in(a_2,a_2+\epsilon)$, the process is almost surely transient. 
\end{thm}
\begin{proof}
It is first shown that there exists a neighborhood of $a_2$ such that the expected hitting time of $x=a_2$ is finite. Similar to Proposition \ref{b2<b1trans}, define $V(x,i)=(1-\alpha c_i)\big(x-a_2)^{\alpha}$. We then we see that for $i\in\{1,2\}$,
$$
LV(x,i)=\bigg\{  \frac{ \dot{Y}(x,i) }{ x-a_2  } + q_{ij} \frac{ c_i -c_j }{ 1 -\alpha c_i}\bigg\}\Big(\alpha V(x,i) \Big).
$$
Note for $x\in(a_2,a_3)$, $\displaystyle \frac{\dot{Y}(x,1) }{  x - a_2} > 0$, $\displaystyle \frac{\dot{Y}(a_3,1) }{  a_3 - a_2} = 0$, and $\displaystyle \frac{\dot{Y}(a_2,1) }{  x - a_2 } \longrightarrow c$ as $x\searrow a_2$ for some $c\in\mathbb{R}$. From this observation we may conclude $\displaystyle \frac{\dot{Y}(x,1) }{  x - a_2}$ is bounded in this interval. Moreover, $\displaystyle \frac{\dot{Y}(a_3, 2) }{  a_3 - a_2} < 0$ and $\displaystyle \frac{\dot{Y}(x,2) }{  x - a_2 } \longrightarrow -\infty$ as $x\searrow a_2$, we may find values of $\alpha$, $c_1$, and $c_2$ where there exists an $\epsilon>0$ such for $x\in(a_2,a_2+\epsilon)$, $LV(x,i)<0$.

To finish the proof, Theorem 3 in Myen and Tweedie \cite{MT93_1} will be applied. Take a constant $\epsilon_0$, where $0<\epsilon_0< \epsilon$, $M$ as the Lebesgue measure, and $\mathfrak{B}$ as the Borel $\sigma$-algebra. Then for $A \in \mathfrak{B}\Big( (0,1) \Big)$, define $\Phi\big( A \big) = M\Big( A \cap (a_2,a_2+\epsilon_0) \Big)$. One may see that the measure $\Phi$ is an irreducible measure (see  \cite{MT93,MT93_1,DMT95} for further information). For $x\in (a_2,a_2+\epsilon_0)$, we may conclude that $P_{x,i}\big( \tau_{a_2+\epsilon_0}^+=\infty \big)>0$.  Therefore, the assumptions of Theorem 3 hold and our dynamic is transient.
\end{proof}

The logic in Proposition \ref{trans-prop} displayed the existence of an $\epsilon$ by utilizing the unboundedness of the term $\displaystyle \frac{\dot{Y}(x,2) }{  x - a_2 }$ at $x=a_2$ for an arbitrary dynamic. Further inspection of the proof shows the term $\displaystyle q_{ij} \frac{ c_i -c_j }{ 1 -\alpha c_i}$ dictates the size of $\epsilon$, yielding the potential of $\epsilon$ to be large enough to cover the size of the interval $(a_2,a_3)$, contingent on the values of $q_{12}$ and $q_{21}$. Application of the same Lyapunov function yields the following corollary.

\begin{cor}\label{full-trans}
For the dynamic defined by Equation \eqref{smr} where $b_2 > b_1$, and given mutation rates $\mu_1$ and $\mu_2$, there exists values of $q_{21}$ and $q_{12}$ where the dynamic is transient in the interval $(a_2 ,a_3)$.
\end{cor}

As the logic in Lemma \ref{b2>b1trans} displayed, considering the case when $b_2 > b_1$, since $(a_2,\hat{a})$ only small changes in the Lyapunov function are needed to derive an analogous result. With this consideration, Proposition \ref{trans-prop} and Lemma \ref{full-trans} yield a foundation for inferring analogous result, encompassed in the following two corollaries.

\begin{cor}\label{b2>b1trans-prop}
For the dynamic defined by Equation \eqref{smr} where $b_2 < b_1$, there exists an $\epsilon>0$ such that for an initial condition $x_0\in( a_2 - \epsilon, a_2)$ the process is almost surely transient. 
\end{cor}

\begin{cor}\label{b2>bfull-trans}
For the dynamic defined by Equation \eqref{smr} where $b_2 < b_1$, and given mutation rates $\mu_1$ and $\mu_2$, there exists values of $q_{21}$ and $q_{12}$ where the dynamic is transient in the interval $(a_1 ,a_2)$.
\end{cor}

\begin{re}
For the type 1 dominate case, $b_2>b_1$, Propositions \ref{trans-prop}, \ref{prop-pos-rec}, and \ref{b2<b1trans} display a counterintuitive property of the an almost sure flow to an invariant neighborhood close to $a_1$, favoring type 2, the dominated subpopulation in the majority of initial conditions.  
\end{re}

The proof in Lemma \ref{full-trans} showed how the values of $q_{21}$ and $q_{12}$ play a role in the evolution of the dynamic. If the interval $(a_2, a_2 + \epsilon)$ in Proposition \ref{trans-prop} is small, one may posit the existence of a subinterval of $(a_2+\epsilon, a_3)$ for the dynamic to be recurrent. The proposition below shows the existence of such a dynamic.

\begin{prop}\label{localrecurrent}
For the dynamic defined by Equation \eqref{smr} where $b_2 > b_1$, and given mutation rates $\mu_1$ and $\mu_2$, there exists values $q_{21}$, $q_{12}$, and $\displaystyle \hat{\epsilon}>0$ dependent on $q_{21}$ and $q_{12}$ where for the initial condition $x_0\in(a_2, a_3)$ the dynamic is recurrent in a strict subinterval of $(a_2+ \hat{\epsilon} ,a_3)$. 
\end{prop}
\begin{proof}
Lemma \ref{a3trans} states the point $x=a_3$ is stochastically unstable. We will use Lemma \ref{a3trans} to prove the proposition by showing the existence of a subinterval in $( a_2 , a_3 )$ where if the initial condition is in this subinterval, the dynamic will almost surely evolve to a neighborhood of $a_3$.  For this subinterval, take an arbitrary $\tilde{\epsilon}>0$ to give an upper bound with the interval $(a_2,a_3 - \tilde{\epsilon} )$, where a value will be later calibrated. For a lower bound of this interval, we will derive values of $q_{21}$ and $q_{12}$ where one may find an $\hat{\epsilon}>0$ where for $\bar{\epsilon}$ the maximum value in Proposition \ref{trans-prop},  $\hat{\epsilon}>\bar{\epsilon}$, and if $x\in( a_2+\hat{\epsilon} , a_3 -\tilde{\epsilon} )$ and $i\in\{1,2\}$, then $E_{x,i}\big[ \tau_{ a_3-\tilde{\epsilon} }^{-}]<\infty$. 

Define the Lyapunov function $V(x,i)=(1-\alpha c_i)\big(a_3-x)^{\alpha}$. One may then see that for $i\in\{1,2\}$,
$$
LV(x,i)=\bigg\{  \frac{ - \dot{Y}(x,i) }{ a_3 - x  } + q_{ij} \frac{ c_i -c_j }{ 1 -\alpha c_i}\bigg\} \Big( \alpha V(x,i) \Big).
$$

For an arbitrary $\epsilon>0$, notice for $x\in(a_2 + \epsilon, a_3- \tilde{\epsilon} )$, there exists values $k_1<0$ and $k_2>0$ such that  $k_1<\displaystyle \frac{- \dot{Y}(x,1) }{  a_3 - x} < 0$ and $0 < \displaystyle \frac{-  \dot{Y}(x,2)  }{  a_3-x  } < k_2$. Given the boundedness and continuity of the two functions on this interval, one may find values of $q_{21}$, $q_{12}$, $\alpha$, $c_1$, $c_2$, and $\tilde{\epsilon}$ where $LV(x,i)<0$ for $x\in(a_2+\hat{\epsilon}, a_3 -\tilde{\epsilon})$, $i=1,2$, and $a_3 - (a_2+\hat{\epsilon}) > \tilde{\epsilon}$. 

This logic only holds if it is not the case that $\bar{\epsilon} \geq a_3-a_2$. Since the mutation rates are fixed, the proof in Proposition \ref{trans-prop} displays that $\bar{\epsilon}$ is contingent on the values of $q_{21}$ and $q_{12}$ and varies in size. Therefore, one may calibrate the values of $q_{21}$ and $q_{12}$ to derive an  $\bar{\epsilon} >0$ such that $a_3-a_2 > \bar{\epsilon} $.

Following the proof in Proposition \ref{prop-pos-rec} finishes the statement.
\end{proof}

When $b_2 < b_1$, a similar logic gives the existence of a subinterval in $(a_1, a_2-\epsilon)$ for the dynamic to hold the recurrent property, which natural yields the corollary below.

\begin{cor}\label{b2>b1localrecurrent}
For the dynamic defined by Equation \eqref{smr} where $b_2 < b_1$, and given mutation rates $\mu_1$ and $\mu_2$, there exists values $q_{21}$, $q_{12}$, and $\displaystyle \hat{\epsilon}>0$ dependent on $q_{21}$ and $q_{12}$ where for the initial condition $x_0\in(a_1, a_2)$ the dynamic is recurrent in a strict subinterval of $(a_1, a_2- \hat{\epsilon})$. 
\end{cor}

Considering the case when $b_2 > b_1$, the combinations of Propositions \ref{prop-pos-rec}, \ref{trans-prop}, and \ref{localrecurrent} display the long-run behavior of the dynamic to be contingent on the initial condition, showing that values above a certain point will almost surely evolve in a strict subinterval of $(a_2,a_3)$, and values below this point will almost surely evolve in a subinterval of $(a_1,\hat{a})$. Although the local long-run behavior yields invariant subintervals, giving a strong recurrent property, the global long-run behavior of the dynamic displays values in the state space which are almost-surely unattainable, suggesting a weaker recurrent property. Corollaries \ref{cor-pos-rec}, \ref{b2>b1trans-prop}, and \ref{b2>b1localrecurrent} give a similar analysis when $b_2 < b_1$.

Corollary \ref{full-trans} shows the existence of conditions for the dynamic to be transient in the interval $(a_2,a_3)$, coupled with Propositions \ref{prop-pos-rec} and \ref{trans-prop}, displays a long-run behavior not contingent on the initial condition, suggesting a strong recurrence property. This intuition is formally shown below.

\begin{thm}\label{pos_rec}
For the dynamic defined by Equation \eqref{smr} where $b_2 > b_1$ and the properties in Corollary \ref{full-trans} hold, the process is positive recurrent. 
\end{thm}
\begin{proof}
To show the property, two points at opposite ends of the unit interval will be taken then displayed to eventually evolve in the same neighborhood. Throughout the proof, $i \in \{1,2 \}$. 

For initial condition $x>a_3$, Lemma \ref{1trans}, Corollary \ref{full-trans}, and Lemma \ref{b2<b1trans} yields the existence of an $\epsilon_1>0$ where $E_{x,i}[ \tau_{a_1 + \epsilon_1 }^+]<\infty$. For initial condition $x < a_1$, Lemma \ref{0trans} and Proposition \ref{prop-pos-rec} yield the existence of an $\epsilon_2 >0$ where $E_{x,i}[ \tau_{a_1 + \epsilon_2 }^-]<\infty$. Proposition \ref{prop-pos-rec} and the strong Markov property tells us that the process is positive recurrent in a subinterval of $(a_1,\hat{a})$.
\end{proof}

\begin{cor}\label{b2>b1pos_rec}
For the dynamic defined by Equation \eqref{smr} where $b_2 < b_1$ and holds the properties in Corollary \ref{b2>bfull-trans}, then process is positive recurrent. 
\end{cor}

\subsection{The Case of Equality}\label{The Case of Equality}

Recall from Section \ref{MR} that when $b_1 = b_2$ the equality $\hat{a}=a_2$ holds. As a consequence of this equality, the dynamic will slowly evolve to and quickly evolve away from this point. This property was used to show that the point $a_3$ is unstable in Lemma \ref{a3trans}, and the point $a_1$ is unstable in Proposition \ref{prop-pos-rec}. Moreover, since the characteristics of the dynamic with the equality $b_1 = b_2$ at the ends of the unit interval are similar to the properties shown of the dynamic when $b_1 \neq b_2$.

The equality $\hat{a}=a_2$ yields that the dynamic of the process to be completely dependent on the initial condition, which in turn dictates the subinterval where the process holds the recurrent property. 

\begin{thm}\label{propb1=b2}
For the dynamic defined by Equation \eqref{smr} where $b_1 = b_2$, then contingent on the initial condition $\displaystyle x_0\in[0,\hat{a} ) \cup(\hat{a} ,1]$ and $i \in \{1,2\}$, the process will eventually evolve in the set $(a_1, \hat{a} )$, or $(\hat{a} , a_3)$ and will be recurrent in either subinterval. 
\end{thm}
\begin{proof}
Lemma \ref{0trans} and Lemma \ref{1trans} yields that the process will evolve in either subinterval in finite time. Proposition \ref{prop-pos-rec} finishes the proof.  
\end{proof}

\section{Conclusions}
The dynamic consisted of a previously characterized process evolving on the two-dimensional simplex and a telegraph noise that randomly switches the value of the mutation parameter above and below the bifurcation point. This minor and natural random perturbation created a complex dynamic. The analysis focused on a general view, and for the case of $b_2 \neq b_1$, displayed the influence the parameters in the Markov chain generator have on the evolution. Given the restriction of the values for the mutation term, focusing on $q_{21}$ and $q_{21}$ and the potential behavior the generator values may dictate was natural. 

The analysis of the case of $b_2 \neq b_1$ yielded two long-run behaviors, the first behavior an intuitive positive recurrence property where the dynamic will eventually evolve in a subinterval, and the second surprising behavior, initial conditions dictating which of the two subintervals the dynamic will eventually evolve in. There may be a third behavior, given the initial condition is in the subinterval with the potential to be transient, the dynamic has a positive probability of evolving in either subinterval. Simulations displayed this may be a characteristic, however, these observation may come from the approximation or not running the simulations for a sufficient amount of time. 

\begin{figure}[H]
\centering
\includegraphics[scale=.40]{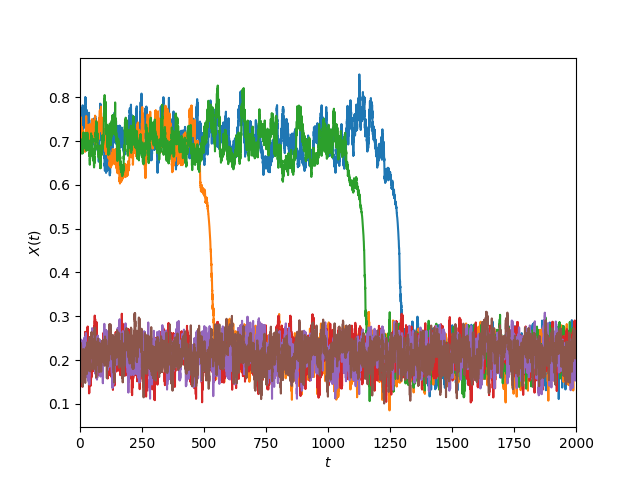}
\includegraphics[scale=.40]{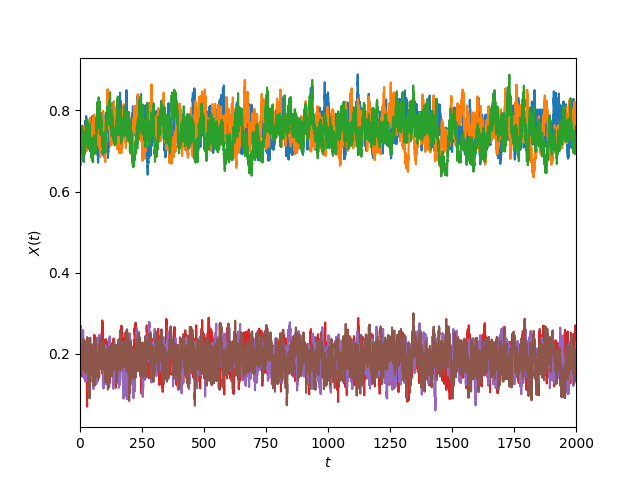}
\caption{The simulations show the evolution of the dynamic with payoffs $b_1 = .2$ and $b_2 = .3$, switched mutation rates $u_1 = .01$ $u_2 = .26$, and two different initial conditions $x_0=.7$ and $x_0=.2$. The left hand side simulation was taken with the Markov generator values $q_{21} = 10.0$ and $q_{12} = 10.0$, and the right hand simulation was taken with Markov generator $q_{21} = 10.0$ and $q_{12} = 12.0$, displaying each of the proven long-run characteristics. }
\end{figure}

An important result not in the manuscript is a closed form solution for this dynamic. In particular, characterizing the invariant measure consisting of explicit bounds for the interval(s) (displaying the mass of this measure). To the knowledge of the author, outside of numerical methods, giving an analytical solution of an invariant measure in general is quite a difficult task. 

\section{Acknowledgements}
The author would like to thank Bahman Gharesifard for direction into this area, George Yin for his knowledge in stochastic stability, and to also thank Troy Day for support in writing this manuscript.

\bibliography{ref_mr}                                                                                                                                               
\bibliographystyle{unsrt}                                                                                                                                                                              
\nocite{*}

\end{document}